\documentclass[11pt]{amsart}
\usepackage{amssymb}
\usepackage{hyperref}

\usepackage{tikz-cd}
\usepackage{xcolor}

\newcommand{\longhookrightarrow}{\lhook\joinrel\longrightarrow}

\newcommand{\Z}{\mathbb{Z}}
\newcommand{\Q}{\mathbb{Q}}
\newcommand{\C}{\mathbb{C}}
\newcommand{\W}{\mathbb{W}}
\renewcommand{\P}{\mathbb{P}}

\newcommand{\cK}{\mathcal{K}}
\newcommand{\cR}{\mathcal{R}}
\newcommand{\cO}{\mathcal{O}}
\newcommand{\cE}{\mathcal{E}}

\newcommand{\Fl}{\mathbf{Q}}
\newcommand{\Pl}{\mathbf{P}}
\newcommand{\Flrat}{\mathbf{Q}^{\mathsf{rat}}}
\newcommand{\Iw}{\mathbf{I}}

\newcommand{\tw}{\tilde{w}}
\newcommand{\tv}{\tilde{v}}

\newcommand{\sph}{\mathsf{sph}}

\newcommand{\heis}{\mathfrak{H}}
\newcommand{\Fun}{\mathfrak{F}_P}
\newcommand{\bFun}{\underline{\mathfrak{F}}_P}

\newcommand{\KFl}{K^{\Iw\rtimes\C^*}(\Fl)}

\newcommand{\KFlrat}{K^{\Iw\rtimes\C^*}(\Flrat)}

\newcommand{\Kfin}{\mathbf{K}}
\newcommand{\Ksph}{\Kfin^{\sph}}

\newcommand{\HH}{\mathbb{H}}

\newcommand{\af}{\mathrm{aff}}

\newcommand{\gch}{\mathrm{ch}}
\newcommand{\ch}{\mathrm{ch}}
\newcommand{\rat}{\mathsf{rat}}

\newcommand{\si}{\frac{\infty}{2}}

\newcommand{\la}{\lambda}
\newcommand{\al}{\alpha}
\newcommand{\be}{\beta}
\newcommand{\ga}{\gamma}
\newcommand{\Ga}{\Gamma}

\renewcommand{\Im}{\mathrm{Im}}
\newcommand{\Mat}{\mathrm{Mat}}

\newcommand{\lsi}{\prec}

\newcommand{\fund}{\omega}

\newcommand{\pol}{\varrho}

\newcommand{\comment}[1]{}

\newcommand{\hHH}{\widehat{\HH}}
\newcommand{\K}{\mathsf{k}}
\renewcommand{\k}{\K}
\newcommand{\ext}{\mathrm{ext}}

\newtheorem{thm}{Theorem}

\newtheorem{lem}[thm]{Lemma}
\newtheorem{prop}[thm]{Proposition}
\newtheorem{cor}[thm]{Corollary}
\theoremstyle{remark}
\newtheorem{rem}[thm]{Remark}
\newtheorem{ex}[thm]{Example}
\newtheorem{dfn}[thm]{Definition}
\numberwithin{thm}{section}

\numberwithin{equation}{section}

\begin{document}

\title[Equivariant $K$-theory of semi-infinite flags as nil-DAHA module]
{Equivariant $K$-theory of the semi-infinite flag manifold as a nil-DAHA module}
\author{Daniel Orr}
\address{
	Department of Mathematics (MC 0123),
	460 McBryde Hall, Virginia Tech,
	225 Stanger St.,
	Blacksburg, VA 24061 USA}
\email{dorr@vt.edu}
\date{\today}
\maketitle

%





\begin{abstract}
The equivariant $K$-theory of the semi-infinite flag manifold, as developed recently by Kato, Naito, and Sagaki, carries commuting actions of the nil-double affine Hecke algebra (nil-DAHA) and a $q$-Heisenberg algebra. The action of the latter generates a free submodule of rank $|W|$, where $W$ is the (finite) Weyl group. We show that this submodule is stable under the nil-DAHA, which enables one to express the nil-DAHA action in terms of $W\times W$ matrices over the $q$-Heisenberg algebra. Our main result gives an explicit algebraic construction of these matrices as a limit from the (non-nil) DAHA in simply-laced type. This construction reveals that multiplication by equivariant scalars, when expressed in terms of the Heisenberg algebra, is given by the nonsymmetric $q$-Toda system introduced by Cherednik and the author.
\end{abstract}

\tableofcontents

\section*{Introduction}
Let $G$ be a connected, simply-connected complex simple algebraic group, with Borel subgroup and maximal torus $B\supset T$. The semi-infinite flag manifold \cite{FF,FM} associated with $G$ is the homogeneous space $\Fl^\rat=G(\cK)/(T(\C)\cdot U(\cK))$ where $\cK=\C((z))$ and $U$ is the unipotent radical of $B$. This variant of the affine flag manifold captures the level-zero representation theory of the untwisted affine Lie algebra associated with $G$ \cite{BF2,K1,KNS,MRY}. Furthermore, the space of quasi-maps $\mathbb{P}^1\to G/B$ into the finite-dimensional flag variety admits a closed embedding into $\Fl^\rat$, and thus semi-infinite flag manifolds are intimately related to quantum $K$-theory of $G/B$~\cite{BF1}.

In \cite{KNS}, an equivariant $K$-group $\KFlrat$ is introduced, where $\Iw\subset G(\cR)$, $\cR=\C[[z]]$, is the Iwahori subgroup and $\C^\times$ acts by loop rotation. There is a major difficulty in applying the usual construction of equivariant algebraic $K$-theory---namely, the Grothendieck group of equivariant coherent sheaves---to $\Flrat$, as this space is an ind-infinite scheme which is not Noetherian. Thus $\KFlrat$ is not the Grothendieck group of a category of coherent sheaves, but is constructed to behave as if it were. We review its definition and basic properties in \S\ref{S:si}. One may object to working with this formal substitute for algebraic $K$-theory, but as a demonstration of its value we mention that $\KFlrat$ has already found applications in the quantum $K$-theory of $G/B$ \cite{K2,LNS}, and in particular to the $K$-theoretic version of Peterson's isomorphism \cite{LLMS}.

The $K$-group $\KFlrat$ has the structure of an $(\HH_0,\heis)$-bimodule, where $\HH_0$ is the nil-double affine Hecke algebra (nil-DAHA) and $\heis$ is a $q$-Heisenberg algebra. Acting on the subset of Schubert classes in $\KFlrat$ indexed by the Weyl group $W$ of $G$, the $q$-Heisenberg algebra $\heis$ generates a free submodule of finite rank $|W|$. We prove (Theorem~\ref{T:main}(1)) that this free $\heis$-module is stable under the nil-DAHA $\HH_0$, giving rise to a homomorphism $\varrho_0$ from $\HH_0$ to the algebra $\Mat_W(\heis)$ of $W\times W$ matrices over $\heis$.

\subsection{Algebraic construction}

Assuming that $G$ is simply-laced, our main result (Theorem~\ref{T:main}) gives a different construction of the homomorphism $\varrho_0 : \HH_0 \to \Mat_W(\heis)$ which is purely algebraic, starting from the polynomial representation of the double affine Hecke algebra. The nonsymmetric $q$-Whittaker function and its symmetries \cite{CO} are ultimately responsible for linking the geometry and with our algebra construction (see \eqref{E:nswhitt} in proof of Theorem~\ref{T:main}). In particular, our result shows that the action of $\HH_0$ on $\KFlrat$, expressed through $\varrho_0$, is by the {\em nonsymmetric} $q$-Toda operators of \cite{CO}.

We also define the ``spherical part'' of $\KFlrat$ (see \S\ref{SS:sph}), which should be regarded as the $(G(\cR)\rtimes \C^\times)$-equivariant $K$-theory of $\Flrat$. By taking the diagonal entry of $\varrho_0$ at the identity element of $W$, we obtain a homomorphism $\varrho_0^{\sph}: \Z[X]^W\to \heis$ corresponding to the spherical nil-affine Hecke algebra action on the spherical part of $\KFlrat$ (Corollary~\ref{C:sph}). Here $\Z[X]^W\subset \HH_0$ is a copy of the representation ring $R(G\times\C^\times)$ inside $\HH_0$. 

A consequence of our construction is that homomorphism $\varrho_0^{\sph}$ coincides with the $q$-Toda system of difference operators \cite{C2,CO}. In particular, our Corollary~\ref{C:sph} is closely related to results of Braverman and Finkelberg \cite{BF2}. The role of $q$-Toda systems in quantum $K$-theory and related geometries goes back to the influential works \cite{GL} and \cite{BF1}. More recently, \cite{Ko,KoZ} give geometric incarnations of type~$A$ $(q,t)$-Macdonald difference operators using the equivariant $K$-theory of quasimaps into the cotangent bundle of the flag variety.

As mentioned above, our results assume that $G$ is simply-laced. While we expect that certain adjustments can be made to extend them to general $G$, our method of proof of Theorem~\ref{T:main} applies only in the simply-laced case. (See the remarks at the beginning of Section~\ref{S:DAHA} for more about this.) Thus we do not attempt to formulate the most general result in this paper.

\subsection{Inverse Pieri-Chevalley formula}

For any $G$-weight $\la\in P$ one has in the usual way an equivariant line bundle $\cO(\la)$ on $\Flrat$. The Pieri-Chevalley formula in $\KFlrat$ expresses the action of multiplication by $\cO(\la)$ on Schubert classes $\{[\cO_{\tw}]\}_{\tw\in W_\af}$, where $W_\af$ is the affine Weyl group, as a combinatorial (in general, infinite) sum:
\begin{align}
[\cO(\la)]\cdot[\cO_{\tw}] &= \sum_{\tv\in W_\af} c_{\tw,\tv}^\la\cdot[\cO_{\tv}]
\end{align}
where $c_{\tw,\tv}^\la\in R(T\times\C^\times)$ are equivariant scalars. The first Pieri-Chevalley formula in $\KFlrat$ was given by \cite{KNS} for dominant $\la\in P_+$, as an infinite sum over semi-infinite Lakshmibai-Seshadri paths. A {\em finite} Pieri-Chevalley formula in $\KFlrat$ for antidominant $\la\in P_-$ was stated and proved in \cite{NOS}.\footnote{By an isomorphism from \cite{K2}, this implies the finiteness of multiplication in the small equivariant quantum $K$-theory of $G/B$. In this setting, the anti-dominant Pieri-Chevalley rule also has a close connection to the shift operator of \cite{IMT}.} A Pieri-Chevalley formula in $\KFlrat$ for arbitrary $\la\in P$, interpolating between the dominant and anti-dominant cases, was found recently by Lenart, Naito, and Sagaki \cite{LNS}.

Our main result directly pertains to the inverse Pieri-Chevalley formula, namely to the expansion:
\begin{align}\label{E:IPC}
e^\la\cdot[\cO_{\tw}] &= \sum_{\substack{\tv\in W_\af\\\mu\in P}} d_{\tw,\tv}^{\la,\mu}\cdot[\cO_{\tv}(\mu)]
\end{align}
where $d_{\tw,\tv}^{\la,\mu}\in R(\C^\times)=\Z[q^{\pm 1}]$ and $e^\la\in R(\Iw\rtimes\C^\times)$ is an equivariant scalar. Theorem~\ref{T:main} gives an algebraic construction of the inverse Pieri-Chevalley formula for arbitrary $\la\in P$, simply because the multiplication by equivariant scalars $e^\la$ is part of the nil-DAHA action. A consequence of Theorem~\ref{T:main} is the finiteness of the inverse Pieri-Chevalley formula in $\KFlrat$ for {\em arbitrary} $\la\in P$; in particular, an immediate coarse observation is that the right-hand side of \eqref{E:IPC} can be expressed as a sum over $\tilde{u}\in W_\af$ less than or equal to the translation element $y^\la$ in the usual Bruhat order (see \S \ref{SS:aff-weyl} for the notation used here).
In future work, we plan to use our construction give a complete and explicit description of the inverse Chevalley rule in type $A$ and to derive implications for the structure of the coefficients $d_{\tw,\tv}^{\la,\mu}$ in general.




\section*{Acknowledgements}

The author is grateful to Michael Finkelberg and Syu Kato for numerous discussions on semi-infinite flag manifolds. He also thanks Satoshi Naito and Daisuke Sagaki for related collaborations and Peter Haskell, Peter Koroteev, Cristian Lenart, Leonardo Mihalcea, Mark Shimozono, and Alex Weekes for helpful discussions. The author was supported by a Collaboration Grant for Mathematicians from the Simons Foundation.

\section{Notation}

\subsection{Root data}
Let $G\supset B\supset T$ be as in the introduction. Denote by $R,Q,P$ (respectively, $R^\vee,Q^\vee,P^\vee$) the (co)roots, (co)root lattice, and (co)weight lattice of $G$. For any root $\al\in P$ let $\al^\vee\in R^\vee$ be its associated coroot.

Let $R=R_+\sqcup R_-$ be the decomposition of $R$ into positive and negative roots determined by $B$. We write $\al>0$ (resp., $\al<0$) to indicate that $\al\in R_+$ (resp., $\al\in R_-$). Let $I$ be a Dynkin index set for $G$ and let $\{\al_i\}_{i\in I},\{\fund_i\}_{i\in I}$ (respectively, $\{\al_i^\vee\}_{i\in I},\{\fund_i^\vee\}_{i\in I}$) be the simple (co)roots and fundamental (co)weights, respectively. Let $W=\langle s_i : i\in I\rangle$ be the Weyl group of $G$, where $s_i=s_{\al_i}$ is the simple reflection through $\al_i$. Let $\ell(w)$ be the length of $w\in W$ with respect to $\{s_i\}_{i\in I}$ and let $w_0$ be the longest element of $W$.

Let $Q^\vee_+=\oplus_{i\in I}\Z_+\al_i^\vee$ and $P_+=\oplus_{i\in I}\Z_+\fund_i$ be the cones of effective coweights and dominant weights, respectively, where $\Z_+=\Z_{\ge 0}$. Let $\leq$ be the partial order on $Q^\vee$ given by $\al\le\be$ if and only if $\be-\al\in Q^\vee_+$. 
For $\la\in P_+$, let $V(\la)$ be the irreducible $G$-module with highest weight $\la$ and let $V(\la)_\mu\subset V(\la)$ be its $\mu$-weight space for any $\mu\in P$. 

We say that a statement depending on $\la\in P$ holds for sufficiently dominant $\la$ if there exists an $M\in\Z_+$ such that the statement is true whenever $\la=\sum_{i\in I}m_i\fund_i$ with $m_i\ge M$ for all $i\in I$.

Let $\Z[P]=R(T)$ be the group algebra of $P$, with basis elements $e^\mu\ (\mu\in P)$ such that $e^{\la+\mu}=e^\la e^\mu$ and $e^0=1$. For any finite-dimensional $T$-module $V$, define its character as $\ch\,V=\sum_{\mu\in P} m_\mu e^\mu\in\Z[P]$ where $m_\mu$ is the dimension of $\mu$-eigenspace of $T$ in $V$.

\subsection{Affine Weyl groups}\label{SS:aff-weyl}
Let $W_\af=W\ltimes Q^\vee$ and $W_\ext=W\ltimes P^\vee$ be the affine and extended affine Weyl groups.
We denote elements $(w,\be)$ of these groups by $wy^\be$, i.e., $w=(w,0)$ and $y^\be=(e,\be)$ where $e\in W$ is the identity element.
The group $W_\af=\langle s_i : i\in I_\af\rangle$ is a Coxeter group where $I_\af=I\sqcup\{0\}$ and
\begin{align*}
s_0 = s_\theta y^{-\theta^\vee}
\end{align*}
with $\theta$ the highest (long) root of $G$.

The group $\Pi=P^\vee/Q^\vee$ acts on $W_\af$ by diagram automorphisms. This can be realized as a subgroup of $W_\ext$ by the elements
\begin{align}
\pi_r = y^{\fund_r^\vee}u_r^{-1}
\end{align}
where $r\in I$ is an index of a minuscule fundamental coweight  (i.e., $\al_r$ appears with coefficient $1$ in $\theta$) and $u_r$ is the shortest element of $W$ sending $\fund_r^\vee$ to the antidominant chamber.

Let $Q_\af=Q\oplus\Z\delta$ be the affine root lattice, which has basis $\{\al_i\}_{i\in I_\af}$ where $\al_0=-\theta+\delta$. Let $P_\af=P\oplus\Z\delta$ be the level-zero affine weight lattice. The affine Weyl group $W_\af$ acts on $P_\af$ as follows:
\begin{align}
wy^\be(\mu+k\delta) &= w(\mu)+(k-\langle\be,\mu\rangle)\delta
\end{align}
where $\langle\,,\,\rangle : Q^\vee\times P\to\Z$ is the canonical pairing.

The set of (untwisted) affine roots is $R_\af=\{\al+k\delta : \al\in R, k\in \Z\}$. We say that an affine root is semi-infinite positive, denoted $\al+k\delta \succ 0$ if $\al\in R_+$; otherwise, $\al+k\delta$ is semi-infinite negative, denoted $\al+k\delta\prec 0$. The reflection through an affine root $\al+k\delta$ is given by $s_{\al+k\delta}=s_\al y^{k\al^\vee}\in W_\af$.

The semi-infinite Bruhat order \cite{L} (see also \cite[\S 2.4 and \S A.3]{INS}) is the partial order $\prec$ on $W_\af$ generated by relations $s_{\al+k\delta}\tw\prec \tw$ if and only if $\tw^{-1}(\al+k\delta)\prec 0$. The resulting poset is graded by the length function $\ell_{\si}(wy^\be)=\ell(w)+\langle\be,2\rho \rangle$, where $2\rho=\sum_{\al\in R_+}\al$.

\subsection{Smash products}
Suppose $S$ is a commutative ring with $1$. For any $S$-algebra $S'$ and any group $\Gamma$ acting by $S$-algebra automorphisms on $S'$, we write $S'\rtimes \Ga$ for the smash product $S' \otimes_S S[\Gamma]$, which is an $S$-algebra with multiplication $(x_1\otimes\ga_1)(x_2\otimes\ga_2)=x_1(\ga_1\cdot x_2)\otimes \ga_1\ga_2$.

In the case when $\Ga$ is an abelian group, written additively, we use exponential notation $\{x^\ga\}_{\ga\in\Ga}$ for the standard basis elements of the group algebra $S[\Ga]$, so that $x^{\ga}x^{\ga'}=x^{\ga+\ga'}$ and $x^0=1$. As we will encounter several instances of such group algebras, and sometimes the same algebra will appear in different contexts, we will use various letters for the base of exponentials (e.g., $x$, $y$, $e$, $X$, $Y$)

\subsection{$q$-Heisenberg algebras}\label{SS:q-heis}
The following special case of smash products will arise frequently. Let $S=\Z[q^{\pm 1}]$
and suppose $A$ and $B$ are abelian groups, written additively, together with a bilinear form $A\times B\to \Z, (a,b)\mapsto\langle a,b\rangle$. Let $S'=S[A]$ with basis $\{x^a\}_{a\in A}$ and let $B$ act on $S'$ by $b\cdot x^a = q^{-\langle a,b\rangle} x^a$. Define $\heis_{A,B}$ to be the smash product $S[A]\rtimes B$. Let $\{y^b\}_{b\in B}$ be the standard $S$-basis of $S[B]$, so that $\heis_{A,B}$ has $S$-basis $\{x^ay^b\}_{a\in A,b\in B}$ where $x^a y^b = q^{\langle a,b \rangle}y^b x^a$.


\subsection{Matrices}
Given a commutative ring $S$ with $1$ and any $S$-algebra $S'$ (not necessarily commutative), let $\Mat_W(S')$ denote the $S$-algebra of $W\times W$ matrices with entries in $S'$.

\section{Semi-infinite flag manifolds}\label{S:si}

In this section we recall the construction of the semi-infinite flag manifold $\Flrat$ due to \cite{FM} and the equivariant $K$-group $\KFlrat$ of \cite{KNS}. Our presentation follows \cite{KNS}, but we elaborate further on some points which our crucial for the present work. Near the end of this section we give a detailed example for $G=SL(2)$.

Consider a $(T\times\C^\times)$-module $V$ with the following properties: if $V=\oplus_{i\in\Z} V_i$ where $\C^\times$ acts in $V_i$ by $q^i$, then each $V_i$ is a finite-dimensional $T$-module and $V_i=0$ for $i$ sufficiently large (or small). For such $V$, we define $\ch\,V=\sum_{i\in\Z} q^i \ch\,V_i$ as an element of $\Z[P]((q^{-1}))$ (or $\Z[P]((q))$). We also define the graded dual $V^*=\oplus_{i\in I}V_{-i}^*$ for such $V$. 

Recall that $\cK=\C((z))$ and $\cR=\C[[z]]$. For any finite-dimensional complex vector space $V$, we write $V((z))=V\otimes_{\C}\cK$ and $V[[z]]=V\otimes_{\C}\cR$. Let $\P(V[[z]]) = (V[[z]]-0)/\C^\times$ regarded as an infinite-type projective scheme with homogeneous coordinate ring $S(V[z]^*)$. For $m\in\Z_{\ge 0}$, let $i_m$ be the embedding $\P(V[[z]])\hookrightarrow\P(V[[z]])$ induced by multiplication of $z^m$ on $V[[z]]$.

Let $\C^\times$ act on $\cK$ by loop rotation, i.e., $a\cdot p(z)=p(a^{-1}z)$ for $a\in \C^\times$ and $p(z)\in\cK$. Let $q\in R(\C^*)$ stand for the weight of $z$ under this action---namely, the class of the representation $q(a)=a^{-1}$.

Let $\Iw\subset G(\cR)$ be the Iwahori subgroup, which is the pre-image of $B$ under the evaluation map $G(\cR)\to G$ at $z=0$. Both $\Iw$ and $G(\cR)$ are $\C^\times$-stable.


\subsection{Semi-infinite flags}

Let $\Fl$ be the infinite-type scheme of \cite{FM} which parametrizes tuples $(\ell_\la)_{\la\in P_+}$ of $\C$-lines in $\prod_{\la\in P_+}\P(V(\la)[[z]])$ satisfying the Pl\"{u}cker equations. Such a collection is uniquely determined by the lines $(\ell_{\fund_i})_{i\in I}$, and the map $(\ell_\la)_{\la\in P_+}\mapsto (\ell_{\fund_i})_{i\in I}$ is the Drinfeld-Pl\"{u}cker embedding
\begin{align*}
\Fl \hookrightarrow \Pl:=\prod_{i\in I}\P(V(\fund_i)[[z]]).
\end{align*}
For any $\be\in Q^\vee_+$, the map $i_\be = \prod_{i\in I} i_{\langle\be,\fund_i\rangle}:\Pl\hookrightarrow\Pl$ restricts to a closed immersion $i_\be : \Fl\hookrightarrow\Fl$.

The {\em semi-infinite flag manifold} $\Flrat$ is the direct limit of the family $\{\Fl_\al\}_{\al\in Q^\vee_+}$ where $\Fl_\al\equiv\Fl$ with respect to the maps
$i_{\al\be} = i_{\be-\al}$ for all $\al\le\be$ in $Q^\vee_+$. Thus $\Flrat$ is an ind-infinite scheme. At the level of $\C$-points, we have $\Flrat=G(\cK)/(T(\C)\cdot U(\cK))$.

For any $\la=\sum_{i\in I}m_i\fund_i\in P$ we have a $(G(\cR)\rtimes \C^\times)$-equivariant (resp. $(G(\cK)\rtimes \C^\times)$-equivariant) line bundle $\cO(\la)$ on $\Fl$ (resp. $\Flrat$) given by the restriction of $\boxtimes_{i\in I}\cO(m_i)$ on $\Pl$ (resp. its limit). Here the $\C^\times$-action on all objects is induced by loop rotation on $\cK$.


\subsection{Equivariant $K$-theory of $\Fl$}

For $f=\sum_{k\ge 0} f_k q^{-k}\in\Z[P][[q^{-1}]]$ where $f_k=\sum_{\nu\in P}c_\nu e^\nu\in \Z[P]$, define $|f|=\sum_{k\ge 0}|f_k| q^{-k}$ where $|f_k|=\sum_{\nu\in P}|c_\nu|e^\nu$ and $|c_\nu|$ is the usual absolute value.

The $K$-group $\KFl$ is defined in \cite{KNS} as the $\Z[P][[q^{-1}]]$-module of formal infinite sums
$\sum_{\lambda\in P} f_\lambda[\cO(\lambda)]$ 
for $f_\lambda\in\Z[P][[q^{-1}]]$ satisfying the absolute convergence criterion
\begin{align}\label{E:abs-conv}
\sum_{\lambda\in P} |f_\lambda|\,\gch\,H^0(\Fl,\cO(\lambda+\mu)) \in\Z_+[P][[q^{-1}]]
\end{align}
for all $\mu\in P$,
modulo the $\Z[P][[q^{-1}]]$-submodule generated by $\sum_{\lambda\in P} f_\lambda\cdot[\cO(\lambda)]$ such that
\begin{align}\label{E:equiv}
\sum_{\lambda\in P} f_\lambda\,\gch\,H^0(\Fl,\cO(\lambda+\mu))=0
\end{align}
for sufficiently dominant $\mu\in P_+$.

Two basic features of $\KFl$ are:

\begin{itemize}

\item The product $[\cO(\nu)]\cdot\sum_{\la\in P}f_\la[\cO(\la)]=\sum_{\la\in P}f_\la[\cO(\la+\nu)]$ induced by tensor product of line bundles is well-defined on $\KFl$ for any $\nu\in P$.

\item For suitable quasicoherent sheaves $\cE$ on $\Fl$ (see \cite[Theorem 5.4]{KNS}), one has a well-defined class $[\cE]$, and the map $\cE\mapsto[\cE]$ is additive in short exact sequences.

\end{itemize}


\subsection{Schubert classes}

For any $\tw\in W_\af$ we have a Schubert variety $\Fl(\tw)\subset\Fl^\rat$ equal to the closure of the $\Iw$-orbit through the $(T\times\C^\times)$-fixed point indexed by $\tw$. Our indexing of fixed points is determined as follows: if $\tw=wy^\be\in W_{\af,+}=W\ltimes Q^\vee_+$, then the corresponding fixed point (in $\Fl$) is the collection of lines $(z^{\langle\be,\la\rangle}V(\la)_{ww_0\la})_{\la\in P}$. We have $\Fl(e)=\Fl$ and $\Fl(\tw)\supset \Fl(\tv)$ (with codimension $\ell_\si(\tv)-\ell_\si(\tw)$) if and only if $\tw\preceq \tv$. Thus $\Fl(\tw)\subset \Fl$ if and only if $\tw\succeq e$ if and only if $\tw\in W_{\af,+}$.

For any $\tw\in W_\af$ and $\lambda\in P$, write $\cO_{\tw}=\cO_{\Fl(\tw)}$ and $\cO_{\tw}(\la)=\cO_{\Fl(\tw)}\otimes\cO(\lambda)$. For $\be\in Q^\vee_+$, we have
\begin{align}
(i_\be)_* \cO_{\tw}(\la) = q^{\langle \be,\la\rangle}\otimes\cO_{\tw y^\be}(\la)
\end{align}
as equivariant sheaves, for all $\tw\in W_\af$ and $\la\in P$, and correspondingly
\begin{align}\label{E:glob-sec-trans}
\ch\,H^0(\Fl(\tw y^\be),\cO(\la))=q^{-\langle\be,\la\rangle}\ch\,H^0(\Fl(\tw),\cO(\la)).
\end{align}

For any $\be\in Q^\vee_+$, the map $i_\be$ induces a homomorphism $(i_\be)_* : \KFl \to \KFl$ of $\Z[P][[q^{-1}]]$-modules such that $(i_\be)_*[\cO(\la)] = q^{\langle \be,\la \rangle}[\cO_{y^\be}(\la)]$ for all $\la\in P$. One easily checks that this map is: (i) well-defined, i.e., it respects convergence \eqref{E:abs-conv} and equivalence \eqref{E:equiv}, and (ii) injective. Moreover, one has $(i_\be)_*[\cO_{\tw}(\la)] = q^{\langle \be,\la \rangle}[\cO_{\tw y^\be}(\la)]$ for any $\tw\in W_{\af,+}$ and $\la\in P$.

The equivariant $K$-theory of $\Flrat$ is defined as 
$$\KFlrat=\Z[P]((q^{-1})) \otimes_{\Z[P][[q^{-1}]]} \varinjlim K_\al $$
where the direct limit of $(K_\al\equiv\KFl)_{\al\in Q^\vee_+}$ is taken with respect to $(i_{\al\be})_*=(i_{\be-\al})_*$ for $\al\le\be$ in $Q^\vee_+$ (just as in the definition of $\Flrat$). Thus in $\KFlrat$ one has
\begin{align}
[\cO_{\tw}(\la)]_{\al} &= q^{\langle\be,\la\rangle}[\cO_{\tw y^\be}(\la)]_{\al+\be}
\end{align}
where $[\cE]_\al$ stands for $[\cE]\in\KFl$ viewed as an element of $K_\al$.

One obtains classes $[\cO_{\tw}(\la)]\in\KFlrat$ for $\tw\in W_\af$ and $\la\in P$ which are well-defined by
\begin{align}\label{E:schub-in-lim}
[\cO_{\tw}(\la)] &= q^{\langle\la,\al\rangle}[\cO_{\tw y^\al}(\la)]_\al\in K_\al
\end{align}
for any $\al\in Q^\vee_+$ such that $\tw y^\al\in W_{\af,+}$.

\begin{dfn}
Define $\Kfin$ to be the $\Z[q^{\pm 1}]$-submodule of $\KFlrat$ generated by the classes $\{[\cO_{\tw}(\lambda)]\}_{\tw\in W_\af,\lambda\in P}$.
\end{dfn}



\begin{lem}\label{L:Kfin-basis}
The classes $\{[\cO_{\tw}(\lambda)]\}_{\tw\in W_\af,\lambda\in P}$ form a $\Z[q^{\pm 1}]$-basis of $\Kfin$.
\end{lem}

\begin{proof}
Suppose $\sum_{\tw,\la}c_{\tw,\la}[\cO_{\tw}(\la)]=0$ where $c_{\tw,\la}\in\Z[q^{\pm 1}]$ and the sum is finite. Without loss of generality we may assume that $\tw\in W_{\af,+}$, so that this equation holds in $\KFl$, and that $\la\in P_+$. This is achieved via \eqref{E:schub-in-lim} for sufficiently large $\al\in Q^\vee_+$ and then by tensoring with $[\cO(\nu)]$ sufficiently dominant $\nu\in P_+$. We may also assume that $c_{\tw,\la}\in\Z[q^{-1}]$ after multiplying by a power of $q^{-1}$.

We may then apply the Pieri-Chevalley formula \cite[Theorem~4]{KNS} in $\KFl$, which gives that $[\cO_{\tw}(\la)]=e^{\tw(\la)}[\cO_{\tw}]+\dotsm$, where $\dotsm$ is a convergent $\Z[P][[q^{-1}]]$-linear combination of $[\cO_{\tv}]$ for $\tv\succ \tw$. Moreover, the classes $\{[\cO_{\tw}]\}_{\tw\in W_{\af,+}}$ are topologically linearly independent over $\Z[P][[q^{-1}]]$ \cite[Proposition~5.8]{KNS}. Choosing $\tw_1$ to be a minimal element with respect to the semi-infinite Bruhat order and such that $c_{\tw_1,\la}\neq 0$ for some $\la$, we deduce that $\sum_\la c_{\tw_1,\la}e^{\tw_1(\la)}=0$, whence $c_{\tw_1,\la}=0$ for all $\la\in P$ (since we assume $\la\in P_+$). By induction, we establish the desired linear independence.
\end{proof}


\subsection{Functional realization}

Let $\Fun$ be the $\Z[P][[q^{-1}]]$-module of functions $\psi : P \to \Z[P][[q^{-1}]]$, with pointwise addition and scalar multiplication. Let $\bFun$ be its quotient by the $\Z[P][[q^{-1}]]$-submodule of functions vanishing on sufficiently dominant $\mu$.

For $\la\in P$ consider the function $\psi_\la\in \Fun$ given by
\begin{align}
\psi_\la(\mu) = \gch\,H^0(\Fl,\cO(\lambda+\mu)).
\end{align}
By \cite[Theorem 5.6]{KNS}, the assignment $\Psi([\cO(\la)])=\psi_\la$ extends to an embedding $\Psi: \KFl \hookrightarrow \bFun$ of $\Z[P][[q^{-1}]]$-modules.

Let $\Fun^\rat=\Z[P]((q^{-1}))\otimes_{\Z[P][[q^{-1}]]}\Fun$ and $\bFun^\rat=\Z[P]((q^{-1}))\otimes_{\Z[P][[q^{-1}]]}\bFun$. We regard elements of $\Fun^\rat$ as functions $\psi: P \to \Z[P]((q^{-1}))$.

For any $\al\in Q^\vee_+$ let $j_\al : \bFun \to \bFun^\rat$ be the map induced by $(j_\al\psi)(\la)=q^{\langle\al,\la\rangle}\psi(\la)$ from $\Fun$ to $\Fun^\rat$. Define $\Psi_\al=j_\al\circ\Psi$ for any $\al\in Q^\vee_+$. One computes that $\Psi_\al=\Psi_\be\circ(i_{\be-\al})_*$ whenever $\al\le\be$ in $Q^\vee_+$. Hence the maps $(\Psi_\al)_{\al\in Q^\vee_+}$ give rise to an embedding $\Psi^\rat:\KFlrat \hookrightarrow \bFun^\rat$. This map satisfies 
\begin{align}\label{E:Psi-extend}
\Psi^\rat([\cO_{\tw}(\la)])(\mu) = \ch\,H^0(\Fl(\tw),\cO(\la+\mu))
\end{align}
for all $\tw\in W_\af$ and $\la,\mu\in P$.


\subsection{Heisenberg}

Tensor product by line bundles $\cO(\la) \ (\la\in P)$ and pushforward under the maps $i_\beta \ (\beta\in Q^\vee_+)$ generate a Heisenberg action on $\KFlrat$.

Define the $q$-Heisenberg algebra $\heis=\heis_{P,Q^\vee}=\Z[q^{\pm 1}][P]\rtimes Q^\vee$ (see \S\ref{SS:q-heis}). The spaces $\Fun^\rat$ and $\bFun^\rat$ are right $\heis$-modules under
\begin{align}
(\psi\cdot x^\nu)(\mu) &= \psi(\mu+\nu)\\
(\psi\cdot y^\be)(\mu) &= q^{-\langle\beta,\mu\rangle}\psi(\mu)
\end{align}
where $\nu\in P$ and $\be\in Q^\vee$.


\begin{prop}
$\KFlrat$ is a right $\heis$-module such that
\begin{align}
[\cO_{\tw}(\la)]\cdot x^\nu &= [\cO_{\tw}(\la+\nu)]\label{E:schub-heis-p}\\
[\cO_{\tw}(\la)]\cdot y^\be &= q^{\langle\beta,\la\rangle}[\cO_{\tw y^{\be}}(\la)]\label{E:schub-heis-q}
\end{align}
for all $\tw\in W_\af$, $\la\in P$ and $\nu\in P$, $\be\in Q^\vee$.
Moreover,
\begin{enumerate}
\item[(i)] $\Psi^\rat : \KFlrat \to \bFun^\rat$ is an $\heis$-module monomorphism, and
\item[(ii)] $\Kfin\subset\KFlrat$ is a free $\heis$-module with basis $\{[\cO_w]\}_{w\in W}$.
\end{enumerate}
\end{prop}

\begin{proof}

The image $\Psi(\KFlrat)\subset \bFun^\rat$ is stable under $\heis$, due to:
\begin{align*}
(\Psi^\rat([\cO(\la)]_\al)\cdot x^\nu)(\mu) &= \Psi_\al([\cO(\la)])(\mu+\nu)\\
&= q^{\langle \al,\mu+\nu\rangle}\Psi([\cO(\la)])(\mu+\nu)\\
&= q^{\langle \al,\mu+\nu\rangle}\Psi([\cO(\la+\nu)])(\mu)\\
&= \Psi_\al(q^{\langle \al,\nu\rangle}[\cO(\la+\nu)])(\mu)\\
&= \Psi^\rat(q^{\langle \al,\nu\rangle}[\cO(\la+\nu)]_\al)(\mu)\\
(\Psi^\rat([\cO(\la)]_{\al})\cdot y^\be)(\mu) &= q^{-\langle \be,\mu\rangle}\Psi^\rat([\cO(\la)]_\al)(\mu)\\
&= q^{-\langle \be,\mu\rangle}\Psi^\rat(q^{\langle \ga,\la\rangle}[\cO_{y^{\ga}}(\la)]_{\al+\ga})(\mu)\\
&=q^{-\langle \be,\mu\rangle}q^{\langle \ga,\la\rangle}q^{\langle\al+\ga,\mu\rangle}\Psi([\cO_{y^\ga}(\la)])(\mu)\\
&=\Psi^\rat(q^{\langle \ga,\la\rangle}[\cO_{y^{\ga}}(\la)]_{\al-\be+\ga})(\mu)
\end{align*}
where we choose $\ga\in Q^\vee_+$ so that $\al-\be+\ga\in Q^\vee_+$. These computations extend to convergent infinite sums. Hence $\KFlrat$ can be made uniquely into an $\heis$-module such that $\Psi$ is an $\heis$-homomorphism. 

To obtain \eqref{E:schub-heis-p} and \eqref{E:schub-heis-q}, we simply apply $\Psi$ to both sides and check that they agree using \eqref{E:glob-sec-trans} and \eqref{E:Psi-extend}.
Finally, the freeness assertion is Lemma~\ref{L:Kfin-basis}.
\end{proof}


\subsection{nil-DAHA}\label{SS:nil-DAHA}
The nil-DAHA $\HH_0$ is the ring defined by generators
\begin{align}
T_i \ (i\in I_\af),\quad X^{\nu} \ (\nu\in P),\quad X^{\pm\delta}
\end{align}
and relations
\begin{align}
&T_i T_j\dotsm = T_j T_i\dotsm \qquad\text{($m_{ij}=|s_is_j|$ factors on both sides)}\\
&T_i(T_i+1) = 0\\
&X^0=1\\
&\text{$X^\delta X^{-\delta}=1$, \ $X^\delta$ central}\\
&X^{\nu}X^{\mu}=X^{\nu+\mu}\\
&T_i X^{\nu} = X^{s_i(\nu)}T_i-\frac{X^{\nu}-X^{s_i(\nu)}}{1-X^{\al_i}}. 
\end{align}
We set $D_i=1+T_i \ (i\in I_\af)$. These elements satisfy the braid relations and $D_i^2=D_i$.

The polynomial representation $\Z[P][q^{\pm 1}]$ of $\HH_0$ is given multiplication operators $X^{\nu+k\delta}\mapsto q^{-k}e^{-\nu}$ and Demazure operators
\begin{align}
D_i &\mapsto (1-e^{\alpha_i})^{-1}(1-e^{\alpha_i}s_i)\\
T_i &\mapsto -(1-e^{-\al_i})^{-1}(1-s_i) 
\end{align}
In this representation we have:
\begin{align}\label{E:leibniz1}
D_i(fg)&= D_i(f)g + s_i(f) T_i(g)\\
\label{E:leibniz2}
D_i(fg)&= e^{-\al_i}T_i(f)g + s_i(f)D_i(g).
\end{align}

\newcommand{\cOb}{\overline{\cO}}
\newcommand{\beb}{\overline{\be}}

Let $\cOb_{wy^\be}(\la)=\cO_{wy^{\beb}}(\la)$ where $\beb=-w_0(\be)$.

By \cite[Prop. 6.4]{KNS} (attributed to unpublished work of Braverman and Finkelberg), $\KFlrat$ is a left $\HH_0$-module such that
\begin{align}
X^{\nu+k\delta}\cdot[\cOb_{y^\al}(\lambda)] &= q^{-k}e^{-\nu}[\cOb_{y^\al}(\lambda)]\\
D_i\cdot (e^\ga[\cOb_{y^\al}(\lambda)]) &= \frac{e^\ga-e^{\al_i}e^{s_i(\ga)}}{1-e^{\al_i}}[\cOb_{y^\al}(\lambda)] \qquad (i\neq 0)\\
D_0\cdot (e^\ga[\cOb_{y^\al}(\lambda)]) &= \frac{e^{\gamma}-e^{s_0(\gamma)}}{1-e^{\alpha_0}}[\cOb_{y^\al}(\lambda)]+e^{s_0(\gamma)}[\cOb_{s_0y^\al}(\lambda)]
\end{align}
By comparison of these formulas with the Heisenberg action, we observe that $\KFlrat$ is an $(\HH_0,\heis)$-bimodule.

The nil-DAHA $\HH_0$ acts on $\Fun^\rat$ and $\bFun^\rat$ via its $P$-pointwise action in the polynomial representation.

\begin{lem}\label{L:KF-HH}
The map $\Psi^\rat : \KFlrat \to \bFun^\rat$ is an $\HH_0$-homomorphism, making it a monomorphism of $(\HH_0,\heis)$-bimodules.
\end{lem}

\begin{proof}(cf. \cite[Proof of Prop. 6.4]{KNS})
Since $\Psi^\rat$ is $\Z[P]((q^{-1}))$-linear, it respects the action $X^\nu \ (\nu\in P)$.

Using the commuting $\heis$-action, it suffices to check the assertion for $D_i \ (i\in I_\af)$ on $e^\ga[\cO(\la)]$. (Note that the $D_i$ are only $\Z((q^{-1}))$-linear.)

One has $T_i\cdot\ch\,H^0(\Fl(y^\al),\cO(\la))=0$ for all $i\neq 0$, $\al\in Q^\vee$, and $\la\in P$, because $H^0(\Fl(y^\al),\cO(\la))$ is a $G$-module. Hence using \eqref{E:leibniz1} we see for $i\neq 0$ that
\begin{align*}
D_i\cdot \Psi^\rat(e^\ga[\cO(\la)])(\mu) &= D_i\cdot (e^{\gamma}\,\gch\,H^0(\Fl,\cO(\lambda+\mu))\\
&=D_i(e^{\gamma})\,\gch\,H^0(\Fl,\cO(\lambda+\mu))+0\\
&=\Psi^\rat(D_i\cdot (e^\ga[\cO(\lambda)]))(\mu)
\end{align*}
for all $\mu\in P$.

More generally, we have by \cite[Theorem A(3)]{K1} that
\begin{align}\label{E:D-H}
D_i\,\ch\,H^0(\cOb_{\tw}(\la))&=
\begin{cases}
\ch\,H^0(\cOb_{s_i\tw}(\la)) & \text{if $s_i \tw\lsi \tw$}\\
\ch\,H^0(\cOb_{\tw}(\lambda)) & \text{if $\tw\lsi s_i\tw$}
\end{cases}
\end{align}
for any $\tw\in W_\af$, $i\in I_\af$, and $\la\in P$.
We have $e\lsi s_i$ for $i\neq 0$ (which gives another way to see the above) and $s_0 \lsi e$. 

For $i=0$ we therefore have
\begin{align*}
&\Psi^\rat(D_0\cdot (e^{\gamma}[\cO(\lambda)]))(\mu)\\
&\quad=\frac{e^{\ga}-e^{s_0(\ga)}}{1-e^{\al_0}}\,\gch\,H^0(\Fl,\cO(\lambda+\mu))+e^{s_0(\gamma)}\,\gch\,H^0(\cOb_{s_0}(\la+\mu))\\
&\quad=e^{-\al_0}T_0(e^\ga)\,\gch\,H^0(\Fl,\cO(\lambda+\mu))+e^{s_0(\gamma)}D_0(\gch\,H^0(\Fl,\cO(\lambda+\mu)))\\
&\quad=D_0(e^\ga\,\ch\,H^0(\Fl,\cO(\la+\mu)))\\
&\quad=D_0\cdot\Psi^\rat(e^\ga[\cO(\lambda)])(\mu)
\end{align*}
by \eqref{E:leibniz2}.
\end{proof}

Lemma~\ref{L:KF-HH} and \eqref{E:D-H} together give
\begin{align}\label{E:D-act}
D_i\cdot[\cOb_{\tw}(\la)] 
&=
\begin{cases}
[\cOb_{s_i\tw}(\la)] &\text{if $s_i\tw\lsi \tw$}\\
[\cOb_{\tw}(\la)] &\text{if $\tw\lsi s_i\tw$}.
\end{cases}
\end{align}
for all $\tw\in W_{\af}$ and $\la\in P$.

Below we will show (see Theorem~\ref{T:main}):
\begin{align}
\label{E:Kfin-stable}
\text{$\Kfin$ is stable under $\HH_0$.}
\end{align}
It is of course immediate from \eqref{E:D-act} that $\Kfin$ is stable under $D_i \ (i\in I_\af)$. The main content of \eqref{E:Kfin-stable} is therefore that $\Kfin$ is stable under $X^\nu \ (\nu\in P)$, which is not immediate.

Granting \eqref{E:Kfin-stable} for now, we have that $\Kfin$ is an $(\HH_0,\heis)$-bimodule which is free as a right $\heis$-module with basis $\{[\cO_w]\}_{w\in W}$. Hence for any $H\in \HH$ there exists a unique $W\times W$ matrix $A_H$ with entries in $\heis$ such that
\begin{align}
H\cdot [\cO_{w}] &= \sum_{v\in W} [\cO_{v}]\cdot (A_H)_{vw}
\end{align}
for all $w\in W$. 
We obtain an algebra homomorphism 
\begin{align}\label{E:rho0}
\varrho_0 : \HH_0\to\Mat_{W\times W}(\heis)
\end{align}
given by $\varrho_0(H)=A_H$.

Our goal in the next sections is to give an algebraic construction of the homomorphism $\varrho_0$.

The matrices $\varrho_0(D_i)$ can be directly computed from \eqref{E:D-act}. For $i\neq 0$ we have $s_iw \prec w$ if and only if $s_iw<w$ in the usual Bruhat order on $W$.
For $i=0$ we have $s_0w\prec w$ if and only if $w^{-1}(-\theta)<0$ if and only if $w<s_\theta w$ in the usual Bruhat order on $W$. Hence
\begin{align}\label{E:Dnot0}
\varrho_0(D_i)_{vw} &= 
\begin{cases}
1 & \text{if $v=s_iw<w$ or $v=w<s_iw$}\\
0 & \text{otherwise}
\end{cases}
\qquad (i\neq 0)\\
\label{E:D0}
\varrho_0(D_0)_{vw} &= 
\begin{cases}
y^{-w_0v^{-1}(\theta)} & \text{if $v=s_\theta w>w$}\\
1 & \text{if $v=w>s_\theta w$}\\
0 & \text{otherwise.}
\end{cases}
\end{align}

\subsection{Example: $G=SL(2)$}

Let $W=\{e,s\}$ be the Weyl group, $\al$ the simple root, and $\fund$ the fundamental weight, with respect to the standard upper triangular Borel subgroup and diagonal torus.

Let $\C^2=V(\fund)$ be the standard representation of $SL(2)$. In this case there are no Pl\"{u}cker equations and we have $\Fl=\P(\C^2[[z]])$. Let $x_k,y_k\in\C^2[z]^*$ be the homogeneous coordinate functions corresponding to $(z^k,0)$ and $(0,z^k)$ in $\C^2[[z]]$, respectively.

The semi-infinite Schubert varieties are given by the following equations
\begin{align*}
\Fl(y^{k\al^\vee}):\ &y_1=x_1=\dotsm=y_{k-1}=x_{k-1}=0\\
\Fl(sy^{k\al^\vee}):\ &y_1=x_1=\dotsm=y_{k-1}=x_{k-1}=y_k=0.
\end{align*}
Hence we have the standard exact sequence
\begin{align*}
0 \to e^\fund \otimes \cO(-\fund) \to \cO_e \to \cO_s \to 0
\end{align*}
Tensoring by $\cO(\fund)$ gives:
\begin{align*}
X^{-\fund} \cdot [\cO_e] &= [\cO_e(\fund)]-[\cO_s(\fund)]\\
&=[\cO_e]\cdot x^\fund-[\cO_s]\cdot x^\fund
\end{align*}
By a similar exact sequence, we obtain
\begin{align*}
X^{\fund} \cdot [\cO_{\Fl(s)}(-\fund)] &= [\cO_{\Fl(s)}]-[\cO_{\Fl(y^{\al^\vee})}]
\end{align*}
and hence
\begin{align*}
X^{-\fund} \cdot [\cO_{\Fl(s)}] &= [\cO_{\Fl(s)}(-\fund)] + X^{-\fund}\cdot[\cO_{\Fl(y^{\al^\vee})}]\\
&=[\cO_s]\cdot x^{-\fund}+X^{-\fund}\cdot([\cO_e]\cdot y^{\al^\vee})\\
&=[\cO_s]\cdot x^{-\fund}+(X^{-\fund}\cdot[\cO_e])\cdot y^{\al^\vee}\\
&=[\cO_s]\cdot x^{-\fund}+([\cO_e]\cdot x^\fund-[\cO_s]\cdot x^\fund)\cdot y^{\al^\vee}.
\end{align*}

Thus we find that
\begin{align}\label{E:X-om}
\varrho_0(X^{-\fund}) &= 
\begin{bmatrix}
x^\fund & x^\fund y^{\al^\vee} \\
-x^\fund & x^{-\fund}-x^\fund y^{\al^\vee}
\end{bmatrix}
\end{align}
which has inverse
\begin{align}\label{E:Xom}
\varrho_0(X^{\fund}) &= 
\begin{bmatrix}
x^{-\fund}-y^{\al^\vee} x^\fund & -y^{\al^\vee}x^\fund  \\
x^\fund & x^\fund
\end{bmatrix}.
\end{align}

\subsection{Dual versions}

To prepare for the algebraic construction of $\varrho_0$, it is convenient to apply a dual twist to all preceding constructions. Let $* : \Z[P]((q^{-1})) \to \Z[P]((q))$ denote the dual map on characters given by $(e^\nu)^*=e^{-\nu}$, $q^*=q^{-1}$. 

Correspondingly, let $\Fun^{\rat*}$ be the $\Z[P]((q))$-module of all functions $\psi:P\to\Z[P]((q))$. Extend the definition of $*$ $P$-pointwise to $*:\Fun^\rat\to\Fun^{\rat*}$. Define $\bFun^{\rat*}$ as the quotient of $\Fun^{\rat*}$ by functions vanishing on sufficiently dominant weights. We have an induced map $*:\bFun^{\rat}\to\bFun^{\rat*}$. 

We also denote the inverse of any of these maps by $*$.

We make $\bFun^{\rat*}$ into a right $\heis$-module as follows:
\begin{align}
(\psi\cdot x^\la)(\mu) &= \psi(\mu-\la)\\
(\psi\cdot y^\be)(\mu) &= q^{\langle\be,\mu\rangle}\psi(\mu)
\end{align}
Then $*$ is compatible with the involutive ring automorphism of $\heis$, also denoted $*$, which is given by: $q^*=q^{-1}$, $(x^\la)^*=x^{-\la}$, $(y^{\be})^*=y^\be$. Then:
\begin{align}
(\psi\cdot h)^* &= \psi^*\cdot h^*
\end{align}
for all $\psi\in \bFun^\rat$ and $h\in \heis$.

We make $\bFun^{\rat*}$ into a left $\HH_0$-module by the $P$-pointwise action of the following operators on $\Z[P]((q))$:
\begin{align}\label{E:HH0-pol-star}
X^\delta \mapsto q,\
X^\nu \mapsto e^\nu,\ T_i &\mapsto -(1-e^{\al_i})^{-1}(1-s_i),\\
D_i &\mapsto (1-e^{-\al_i})^{-1}(1-e^{-\al_i}s_i).\notag
\end{align}
Then $*$ is a left $\HH_0$-module homomorphism.

Observe from \eqref{E:Dnot0} and \eqref{E:D0} that $\varrho_0(D_i)^*=\varrho_0(D_i)$ for all $i\in I_\af$, where $*$ is applied entrywise matrices in $\Mat_W(\heis)$.

\section{DAHA}\label{S:DAHA}

We assume for the rest of this paper that
\begin{align}
\text{$G$ is simply-laced.}
\end{align}
We will at times identify $Q$ with $Q^\vee$ and $P$ with $P^\vee$ by imposing $\al=\al^\vee$ for $\al\in R$. We extend the canonical pairing to $\langle\,,\,\rangle : P\times P \to\Q$. We choose $e\ge 1$ minimal so that $e\langle P,P\rangle\subset 2\Z$. We have $\langle \al,\al\rangle=2$ for all $\al\in R$ and $\langle\be,\be\rangle\in 2\Z$ for all $\be\in Q$.

The constructions of this section and the next can be applied more generally in the setting of twisted affine root data. We have opted for less generality in an effort to keep our notation relatively simple. Moreover, the overlap with the preceding constructions---which were based on an untwisted affine root system---is precisely the case simply-laced root data.

Further details on the constructions of this section in the twisted case can be found in \cite{CO}.

\subsection{DAHA}


Let $\HH$ be the $\Z[t^{\pm 1/2}]$-algebra with generators 
\begin{align}
T_i \ (i\in I_\af),\quad X^\nu \ (\nu\in P),\quad X^{\pm \delta/e}
\end{align}
and relations
\begin{align}
&T_i T_j\dotsm = T_j T_i \dotsm  \qquad\text{($m_{ij}=|s_is_j|$ factors on both sides)}\\
&(T_i-t)(T_i+1) = 0\\
&X^{\delta/e}X^{-\delta/e}=1,\ \text{$X^{\delta/e}$ central}\\
&X^\nu X^\mu=X^{\nu+\mu}\\
&X^0=1\\
&T_i X^\nu - X^{s_i(\nu)} T_i = (t-1)(1-X^{\al_i})^{-1}(X^\nu-X^{s_i(\nu)})\label{E:X-bernstein}
\end{align}
for all $i,j\in I_\af$ and $\nu,\mu\in P$.

We observe that $tT_i^{-1} = T_i - (t-1).$

The group $\Pi=P/Q$ acts by $\Z[t^{\pm 1/2}]$-algebra automorphisms on $\HH$. These are given by $\pi(T_i)=T_j$ where $\pi(\al_i)=\al_j$ and $\pi(X^\nu)=X^{\pi(\nu)}$, $\pi(X^{\delta/e})=X^{\delta/e}$. We will call the smash product $\hHH=\HH*\Pi$ the extended DAHA.

If $\nu\in P_+$, we define $Y^\nu=t^{-l/2}\pi T_{i_1}\dotsm T_{i_l}$ for any reduced expression $\nu=\pi s_{i_1}\dotsc s_{i_l}$ in the extended affine Weyl group $W_\ext$. For $\nu\in P$, we write $\nu=\nu_1-\nu_2$ where $\nu_1,\nu_2\in P_+$ and define $Y^\nu=Y^{\nu_1}(Y^{\nu_2})^{-1}$. This is independent of the choice of $\nu_1,\nu_2$. The $Y^\nu$ satisfy $Y^\nu Y^\mu=Y^{\nu+\mu}$ and $Y^0=1$. For any reduced expression $\nu=\pi s_{i_1}\dotsc s_{i_l}$ we have $Y^\nu=t^{-\sum \epsilon_k/2}\pi T_{i_1}^{\epsilon_1}\dotsm T_{i_l}^{\epsilon_l}$ where
\begin{align*}
\epsilon_k = \begin{cases} +1 &\text{if $\pi s_{i_1}\dotsm s_{i_{k-1}}(\al_{i_k})\prec 0$}\\-1 &\text{if $\pi s_{i_1}\dotsm s_{i_{k-1}}(\al_{i_k})\succ 0$}.\end{cases}
\end{align*}


We also define $Y^{\delta/e}=X^{-\delta/e}$.

\subsection{Duality}

\newcommand{\hvarphi}{\widehat{\varphi}}

The extended DAHA $\hHH$ carries, among other symmetries, a $\Z[t^{\pm 1/2}]$-linear  algebra involutive anti-automorphism $\hvarphi$ which is uniquely determined by $\hvarphi(X^{\delta/e})=X^{\delta/e}=Y^{-\delta/e}$, $\hvarphi(X^\nu)=Y^{-\nu}$, and $\hvarphi(T_i)=T_i$ for all $i\in I$ (see \cite{CO} or \cite{C1} and the references therein).

We will make use of the algebra automorphism $\tau_+$ of $\hHH$ \cite{C1}, which fixes $T_i\ (i\neq 0)$, $X^{\delta/e}$, and $X^\nu\ (\nu\in P)$ and is given on the remaining generators by
\begin{align}
\tau_+ : T_0\mapsto X^{-\al_0}tT_0^{-1},\
\pi_r \mapsto X^{\fund_r-\frac{1}{2}\langle\fund_r,\fund_r\rangle\delta}\pi_r.
\end{align}

Define $\varphi=\hvarphi\circ\tau_+$, which is an anti-automorphism of $\hHH$. One checks that $\hvarphi(\tau_+(T_0))=\tau_+(T_0)=X^{-\al_0}tT_0^{-1}$ and hence
\begin{align}
\varphi : X^\nu \mapsto Y^{-\nu} \ (\nu\in P),\ T_i\mapsto T_i\ (i\neq 0),\ T_0\mapsto X^{-\al_0}tT_0^{-1}.
\end{align}


Let $\HH'=\varphi(\HH)$. Thus $\HH'$ is the $\Z[t^{\pm 1/2}]$-subalgebra of $\HH$ generated by $T_i'$ ($i\in I_\af$) and $Y^\nu$ ($\nu\in P$) and $Y^{\delta/e}=X^{-\delta/e}$, where $T_i'=T_i$ for $i\in I$ and 
$T_0'=X^{-\al_0}tT_0^{-1}$. 
We deduce that $\HH'$ can be presented as the $\Z[t^{\pm 1/2}]$-algebra with generators $T_i'$ ($i\in I_\af$) and $Y^\nu$ ($\nu\in P$) and $Y^{\delta/e}$ and relations
\begin{align}
&T_i' T_j'\dotsm = T_j' T_i' \dotsm  \qquad\text{($m_{ij}=|s_is_j|$ factors on both sides)}\\
&(T_i'-t)(T_i'+1) = 0\\
&Y^{\delta/e}Y^{-\delta/e}=1,\ \text{$Y^{\delta/e}$ central}\\
&Y^\nu Y^\mu=Y^{\nu+\mu}\\
&Y^0=1\\
&T_i' Y^\nu - Y^{s_i(\nu)} T_i' = (t-1)(1-Y^{-\al_i})^{-1}(Y^\nu-Y^{s_i(\nu)})\label{E:Y-bernstein}
\end{align}
for all $i,j\in I_\af$ and $\nu,\mu\in P$.



The map $\varphi$ restricts to an anti-isomorphism $\varphi: \HH \to \HH'$ given by $T_i\mapsto T_i'$ ($i\in I_\af$) and $X^\nu\mapsto Y^{-\nu}$ ($\nu\in P$), $X^{\delta/e}\mapsto Y^{-\delta/e}$.

\subsection{Polynomial representation} 
Let $\k=\Q(q^{1/e},t^{1/2})$. The group algebra $\k[P]$, with $\k$-basis $\{x^\la\}_{\la\in P}$, is a left $\hHH$-module such that
\begin{align}
T_i(f)&=ts_i(f) + (t-1)\frac{f-s_i(f)}{1-x^{\alpha_i}} \quad (i\in I_\af)\\
X^\nu(f) &= x^\nu f \quad (\nu\in P)\\
X^{\delta/e}(f) &= q^{1/e}f
\end{align}
and $\Pi$ acts by its group action on $\k[P]$: $\pi(x^\la)=x^{\pi(\la)}$.
Thus a general element $H\in\hHH$ acts on $\K[P]$ by an operator given by a finite sum (a difference-reflection operator):
\begin{align}
H\mapsto \sum_{\tw\in W_\ext}h_{\tw} \tw
\end{align}
where $h_{\tw}\in\k(P)=\mathrm{Frac}(\k[P])$, viewed as a multiplication operator. This gives an embedding 
\begin{align}
\hHH\overset{\pol}{\longhookrightarrow}\k(P)\rtimes W_\ext
\end{align}
whose image leaves $\K[P]$ stable.

\subsection{Functions on $W$}

\newcommand{\kFun}{\k(P)_W}
\newcommand{\End}{\mathrm{End}}
\newcommand{\Diff}{\mathfrak{D}}

Let $\kFun$ be the $\k$-algebra functions $\phi:W \to \K(P)$ under pointwise addition and multiplication. We make $\kFun$ a $W_\ext$-module as follows:
\begin{align}
(u \cdot \phi)(v) &= \phi(u^{-1}v)\\
(y^\eta\cdot \phi)(v) &= y^{v^{-1}(\eta)}\cdot \phi(v)
\end{align}
where $u,v\in W$ and $\mu\in P$. This action is by $\k$-algebra automorphisms, thus making $\kFun$ a module over $\kFun\rtimes W_\ext$ (letting $\k(P)_W$ act on itself by multiplication).


For any $D\in \kFun \rtimes W_\ext$ there exist unique $A_{vw}\in\K(P)\rtimes P \ (v,w\in W)$ such that
\begin{align}
(D\cdot \phi)(v) &= \sum_{v\in W} A_{vw}\phi(w)
\end{align}
for all $\phi\in \kFun$. The assignment $D\mapsto A=(A_{vw})_{v,w\in W}$ is an embedding of $\k$-algebras 
\begin{align}\label{E:Fun-Wext-mat}
\kFun \rtimes W_\ext\longhookrightarrow \Mat_{W}(\k(P)\rtimes P).
\end{align}





\begin{ex}
For $D=\chi\otimes y^\eta u\in\kFun\rtimes W_\ext$, $u,v\in W$, we have
\begin{align*}
((\chi\otimes y^\eta u) \cdot \phi)(v) &=\chi(v) y^{v^{-1}(\eta)}\cdot \phi(u^{-1}v)
\end{align*}
and hence the corresponding matrix is given by
\begin{align*}
A_{vw} &= \begin{cases}\chi(v)y^{v^{-1}(\eta)} &\text{if $w=u^{-1}v$}\\0 &\text{otherwise.}\end{cases}
\end{align*}
\end{ex}


\subsection{Matrix realization of $\HH'$}
For $f\in\K(P)$, define $\phi_f : W \to\K(P)$ by $\phi_f(v)=v^{-1}\cdot f$.
The map $f\mapsto \phi_f$ is $W_\ext$-equivariant. Let
\begin{align}\label{E:kP-kPW}
\k(P)\rtimes W_\ext &\longhookrightarrow \k(P)_W\rtimes W_\ext\\
f\otimes y^\eta u &\longmapsto \phi_f \otimes y^\eta u\notag
\end{align}
be the induced map.



\begin{dfn}\label{D:varrho'}
Let $\varrho':\HH'\hookrightarrow \Mat_{W}(\K(P)\rtimes P)$ be the restriction to $\HH'$ of the following composite map:
\begin{align*}
\hHH \overset{\pol}{\longhookrightarrow} \K(P) \rtimes W_\ext \overset{\eqref{E:kP-kPW}}{\longhookrightarrow} \kFun \rtimes W_\ext \overset{\eqref{E:Fun-Wext-mat}}{\longhookrightarrow} \Mat_{W}(\K(P)\rtimes P).
\end{align*}
\end{dfn}

\begin{ex}
Under the arrows in Definition~\ref{D:varrho'}, $f \otimes y^\eta u\in \k(P)\rtimes W_\ext$ is sent to the matrix $A=(A_{vw})_{v,w\in W}$ given by
\begin{align*}
A_{vw} &= \begin{cases}(v^{-1}\cdot f)y^{v^{-1}(\eta)} &\text{if $w=u^{-1}v$}\\0 &\text{otherwise.}\end{cases}
\end{align*}
\end{ex}

\begin{ex} For any $i\neq 0$ and any $\nu\in P$ we have
\begin{align}
\varrho'(T_i')_{vw} &=
\begin{cases}
\frac{1-t}{x^{v^{-1}(\alpha_i)}-1} & \text{if $v=w$}\\
\frac{tx^{v^{-1}(\alpha_i)}-1}{x^{v^{-1}(\alpha_i)}-1} &\text{if $v=s_iw$}\\
0 & \text{otherwise}
\end{cases}\\
\varrho'(X^{\nu})_{vw} &=
\begin{cases}
x^{v^{-1}(\nu)} & \text{if $v=w$}\\
0 & \text{otherwise.}
\end{cases}
\end{align}
\end{ex}

\begin{ex} Let $G=SL(2)$ and label rows and columns of matrices by $(e,s_1)$. We have:
\begin{align}
\varrho'(T_1') &= 
\begin{bmatrix}
\frac{1-t}{x^{\al_1}-1} & \frac{tx^{\al_1}-1}{x^{\al_1}-1}\\
\frac{tx^{-\al_1}-1}{x^{-\al_1}-1} & \frac{1-t}{x^{-\al_1}-1}\\
\end{bmatrix}\\
\varrho'(T_0') &= 
q^{-1}\begin{bmatrix}
x^{\al_1} & 0 \\
0 & x^{-\al_1}
\end{bmatrix}
\begin{bmatrix}
\frac{t-1}{q^{-1}x^{\al_1}-1} & \frac{tqx^{-\al_1}-1}{x^{-\al_1}-1}y^{\al_1} \\
\frac{tqx^{\al_1}-1}{x^{\al_1}-1}y^{-\al_1} & \frac{t-1}{q^{-1}x^{-\al_1}-1}
\end{bmatrix}\\
\varrho'(Y^{\fund_1})&=
t^{-1/2}
\begin{bmatrix}
y^{\fund_1} & 0\\
0 & y^{-\fund_1}
\end{bmatrix}
\begin{bmatrix}
\frac{tx^{-\al_1}-1}{x^{-\al_1}-1} & \frac{t-1}{1-x^{-\al_1}}\\
\frac{t-1}{1-x^{\al_1}} & \frac{tx^{\al_1}-1}{x^{\al_1}-1} \\
\end{bmatrix}
\end{align}
Here we use that $Y^{\fund_1}=t^{-1/2}\pi T_1$, where $\pi=y^{\fund_1}s_1$, so that in the polynomial representation
\begin{align}
Y^{\fund_1} \mapsto t^{-1/2}y^{\fund_1}\left(\frac{tx^{-\al_1}-1}{x^{-\al_1}-1}+\frac{t-1}{1-x^{\al_1}}s_1\right).
\end{align}
\end{ex}


\section{nil-DAHA}

\subsection{nil-DAHA's}

\newcommand{\lHH}{\widetilde{\HH}}

Let $\lHH$ and $\lHH'$ denote the $\Z[t]$-subalgebras of $\HH$ and $\HH'$ which are generated by $T_i$ ($i\in I_\af$), $X^\nu$ ($\nu\in P$), $X^{\pm\delta}$ and $T_i'$ ($i\in I_\af$), $Y^\nu$ ($\nu\in P$), $Y^{\pm\delta}$, respectively.

The corresponding nil-DAHA's $\HH_0$ and $\HH_0'$ are obtained by specializing $t=0$, i.e., $\HH_0=\lHH/t\lHH$ and $\HH_0'=\lHH'/t\lHH'$. Thus $\HH_0$ admits the presentation of \S\ref{SS:nil-DAHA} and $\HH_0'$ can be presented as follows:


$\HH_0'$ is the ring with generators 
\begin{align}
T_i' \ (i\in I_\af),\quad Y^\nu \ (\nu\in P), \quad Y^{\pm \delta}
\end{align}
and relations
\begin{align}
&T_i' T_j'\dotsm = T_j' T_i' \dotsm  \qquad\text{($m_{ij}=|s_is_j|$ factors on both sides)}\\
&T_i'(T_i'+1) = 0\\
&Y^{\delta}Y^{-\delta}=1,\ \text{$Y^{\delta}$ central}\\
&Y^\nu Y^\mu=Y^{\nu+\mu}\\
&Y^0=1\\
&T_i' Y^\nu - Y^{s_i(\nu)} T_i' = -(1-Y^{-\al_i})^{-1}(Y^\nu-Y^{s_i(\nu)})\label{E:Y-bernstein}
\end{align}
for all $i,j\in I_\af$ and $\nu,\mu\in P$.


We set $D_i'=T_i'+1$ for $i\in I_\af$.

We obtain a ring anti-automorphism $\varphi_0:\HH_0\to\HH_0'$ as the specialization of $\varphi$.


\subsection{Ruijsenaars-Etingof limit}
Following \cite[\S4.4]{CO}, we define the Ruij\-senaars-Etingof limit
\begin{align}
\varrho_0'(H) &= \lim_{t\to 0} \varkappa(\varrho'(H))
\end{align}
for $H\in\HH'$, when it exists, where $\kappa$ is the automorphism of $\K(P)\rtimes P$ given by
\begin{align}
\varkappa(x^\lambda) &= t^{-\langle\la,\rho\rangle}x^\la\\
\varkappa(y^\mu) &= t^{\langle\mu,\rho\rangle}y^\mu
\end{align}
and acting entrywise on $\Mat_{W}(\K(P)\rtimes P)$.

Let $\heis'=\heis_{Q^\vee,P}=\Z[q^{\pm 1}][Q^\vee]\rtimes P$, with generators $x^\be \ (\be\in Q^\vee)$ and $y^\la \ (\la\in P)$ such that $x^\be y^\la=q^{\langle \be,\mu\rangle}y^\la x^\be$.

\begin{thm}
For any $H\in\lHH'$, the limit $\varrho_0'(H)$ exists and belongs to $\Mat_{W}(\heis')$, giving a homomorphism $\varrho_0':\HH'_0 \to \Mat_{W\times W}(\heis')$.
\end{thm}

\begin{proof}
The existence of $\varrho_0'(H)$ for $H\in\lHH'$ is an immediate consequence of \cite[Theorem 4.1(i)]{CO}. It follows readily from the construction that, if $\varrho_0'(H)$ exists, it must belong to $\Mat_{W}(\heis')$.
\end{proof}

\begin{ex} For $G=SL(2)$, we have: 

\begin{align}
\notag
\varkappa(\varrho'(T_1')) &= 
\begin{bmatrix}
\frac{1-t}{t^{-1}x^{\al_1}-1} & \frac{x^{\al_1}-1}{t^{-1}x^{\al_1}-1}\\
\frac{t^2x^{-\al_1}-1}{tx^{-\al_1}-1} & \frac{1-t}{tx^{-\al_1}-1}
\end{bmatrix}
\end{align}
\begin{align}
\varrho_0'(T_1')&=
\begin{bmatrix}
0 & 0 \\
1 & -1
\end{bmatrix},\qquad
\varrho_0'(D_1') =
\begin{bmatrix}
1 & 0 \\
1 & 0
\end{bmatrix}
\end{align}


\begin{align}
\notag
\varkappa(\varrho'(T_0')) &= 
q^{-1}\begin{bmatrix}
t^{-1}x^{\al_1} & 0 \\
0 & tx^{-\al_1}
\end{bmatrix}
\begin{bmatrix}
\frac{t-1}{q^{-1}t^{-1}x^{\al_1}-1} & \frac{t^2qx^{-\al_1}-1}{tx^{-\al_1}-1}ty^{\al_1} \\
\frac{qx^{\al_1}-1}{t^{-1}x^{\al_1}-1}t^{-1}y^{-\al_1} & \frac{t-1}{q^{-1}tx^{-\al_1}-1}
\end{bmatrix}\\
\notag
&=q^{-1}\begin{bmatrix}
t^{-1}x^{\al_1}\frac{t-1}{q^{-1}t^{-1}x^{\al_1}-1} & t^{-1}x^{\al_1}\frac{t^2qx^{-\al_1}-1}{tx^{-\al_1}-1}ty^{\al_1} \\
tx^{-\al_1}\frac{qx^{\al_1}-1}{t^{-1}x^{\al_1}-1}t^{-1}y^{-\al_1} & tx^{-\al_1}\frac{t-1}{q^{-1}tx^{-\al_1}-1}
\end{bmatrix}
\end{align}
\begin{align}
\varrho_0'(T_0') &=
\begin{bmatrix}
-1 & q^{-1}x^{\al_1}y^{\al_1} \\
0 & 0
\end{bmatrix},\qquad
\varrho_0'(D_0') =
\begin{bmatrix}
0 & q^{-1}x^{\al_1}y^{\al_1} \\
0 & 1
\end{bmatrix}
\end{align}
\begin{align*}
\varkappa(\varrho'(Y^{\fund_1}))&=
\begin{bmatrix}
y^{\fund_1} & 0\\
0 & t^{-1}y^{-\fund_1}
\end{bmatrix}
\begin{bmatrix}
\frac{t^2x^{-\al_1}-1}{tx^{-\al_1}-1} & \frac{t-1}{1-tx^{-\al_1}}\\
\frac{t-1}{1-t^{-1}x^{\al_1}} & \frac{x^{\al_1}-1}{t^{-1}x^{\al_1}-1} \\
\end{bmatrix}
\end{align*}
\begin{align}
\varrho_0'(Y^{\fund_1})&=
\begin{bmatrix}
y^{\fund_1} & -y^{\fund_1}\\
y^{-\fund_1}x^{-\al_1} & y^{-\fund_1}(1-x^{-\al_1}) \\
\end{bmatrix}
\end{align}
\end{ex}

\begin{ex}\label{EX:Ti}
In general, for $i\neq 0$ we have
\begin{align}
\varrho_0'(T_i')_{vv} =
\begin{cases}
0 &\text{if $v^{-1}(\al_i)>0$}\\
-1 &\text{if $v^{-1}(\al_i)<0$}
\end{cases}
\end{align}

\begin{align}
\varrho_0'(T_i')_{v,s_iv} =
\begin{cases}
0 &\text{if $v^{-1}(\al_i)>0$}\\
1 &\text{if $v^{-1}(\al_i)<0$}
\end{cases}
\end{align}
and all other entries of $\varrho_0'(T_i')$ vanish. For $i=0$ we have
%
\begin{align}
\varrho'(T_0')_{vw} &= 
\begin{cases}
q^{-1}x^{v^{-1}(\theta)}\frac{t-1}{q^{-1}x^{v^{-1}(\theta)}-1} & \text{if $v=w$}\\
q^{-1}x^{v^{-1}(\theta)}\frac{tqx^{-v^{-1}(\theta)}-1}{qx^{-v^{-1}(\theta)}-1}y^{v^{-1}(\theta)} & \text{if $v=s_\theta w$}\\
0 & \text{otherwise}
\end{cases}
\end{align}
and hence
\begin{align}
\varrho_0'(T_0')_{vv} &=
\begin{cases}
-1 & \text{if $v^{-1}(\theta)>0$} \\
0  & \text{if $v^{-1}(\theta)<0$}
\end{cases}\\
\varrho_0'(T_0')_{v,s_\theta v} &=
\begin{cases}
q^{-1}x^{v^{-1}(\theta)}y^{v^{-1}(\theta)} & \text{if $v^{-1}(\theta)>0$}\\
0 & \text{if $v^{-1}(\theta)<0$}
\end{cases}
\end{align}
and all other entries of $\varrho_0'(T_0')$ vanish.
\end{ex}




\section{Main Theorem}


We will need the $\Z[q^{\pm 1}]$-linear anti-isomorphism $\tau : \heis' \to \heis$ given by
\begin{align}
\tau:\  x^\be\mapsto q^{\frac{\langle \be,\be\rangle}{2}}x^{-w_0(\be)}y^{w_0(\be)}, \quad y^\mu\mapsto x^{w_0(\mu)}.
\end{align}
We extend the definition of $\tau : \Mat_W(\heis') \to \Mat_W(\heis)$ to an anti-isomorphism of matrix algebras by applying it entrywise and taking the matrix transpose.

For any (left or right) $\heis'$-module $M$, let $M_W$ be the set of all functions $\phi:W\to M$, made into a (left or right) $\Mat_W(\heis')$-module as follows:
\begin{align}
(A\cdot \phi)(v) &= \sum_{w\in W} A_{vw}\phi(w)\\
(\phi\cdot A)(w) &= \sum_{v\in W} \phi(v)A_{vw}
\end{align}
thinking of $\phi$ as a column or row vector, respectively.

\begin{thm}\label{T:main}
\begin{enumerate}
\item $\Kfin$ is an $\HH_0$-module, giving rise to $\varrho_0$ of \eqref{E:rho0}.
\item The following diagram is commutative:
\begin{center}
\begin{tikzcd}
\HH_0' \arrow[r, "\varphi_0^{-1}"] \arrow[d,"\varrho_0'"]
& \HH_0 \arrow[d, "\varrho_0^*"] \\
\Mat_W(\heis') \arrow[r, "\tau"]
& \Mat_W(\heis)
\end{tikzcd}
\end{center}
\end{enumerate}
\end{thm}

\begin{rem}
It is instructive to verify part (2) of Theorem~\ref{T:main} directly on the elements $T_i'$ for $i\in I_\af$, by comparing Example~\ref{EX:Ti} with \eqref{E:Dnot0} and \eqref{E:D0}.
\end{rem}

\begin{proof}
In the space $(\bFun^{\rat*})_W$ consider the following special function $\psi$ whose values $\psi(w)$ are the following elements of $\bFun^{\rat*}$:
\begin{align}\label{E:nswhitt}
\psi(w)=\left(\mu\mapsto \ch\,H^0(\Fl(w),\cO(\mu))^*\right).
\end{align}
By \cite{K1} and \cite{FMO} we know that (see also \cite{KNS,N}):
\begin{align}
H^0(\Fl(w),\cO(\mu))^*
\cong
\begin{cases}
\W_{w(w_0\mu)} & \text{if $\mu\in P^+$}\\
0 & \text{otherwise}
\end{cases}
\end{align}
as $\mathrm{Lie}(\Iw\rtimes\C^\times)$-modules, where $\W_{w(\nu)}$ for $w\in W$ and $\nu\in P_-$ is the global generalized Weyl module of \cite{FMO}. Then \cite[Theorem B]{FMO} asserts that
\newcommand{\tpsi}{\widetilde{\psi}}
\begin{align*}
\tpsi(w)&=(\ga'\ga)^{-1}\sum_{\mu\in P_-}q^{\frac{\langle\mu,\mu\rangle}{2}}x^{-\mu}\ch\,\W_{w(\mu)}\\
&=(\ga'\ga)^{-1}\sum_{\mu\in P_+}q^{\frac{\langle\mu,\mu\rangle}{2}}x^{-w_0\mu}\ch\,\W_{w(w_0\mu)}
\end{align*}
is the nonsymmetric $q$-Whittaker functions, i.e, it satisfies the nonsymmetric $q$-Toda equations:
\begin{align}
\tau_+(H)\cdot\tpsi = \varrho_0'(\varphi_0(H))\cdot\tpsi 
\end{align}
for all $H\in\HH_0$. Here
\begin{align}
\ga'=\sum_{\mu\in P} q^{\frac{\langle\mu,\mu\rangle}{2}}x^\mu,\quad \ga=\sum_{\mu\in P} q^{\frac{\langle\mu,\mu\rangle}{2}}e^\mu
\end{align}
are Gaussians and $\HH_0$ is acting $W$-pointwise and by the polynomial representation \eqref{E:HH0-pol-star} on $\Z[q^{\pm 1}][e^\nu : \nu\in P]$ (which contains $\ch\,\W_{w(\mu)}$). We observe that $\ga^{-1} H\ga=\tau_+(H)$ in this representation.

The nonsymmetric $q$-Toda equations for $\tpsi$ translate to the following equations for $\psi$:
\begin{align}
H\cdot\psi = \psi \cdot \tau(\varrho_0'(\varphi_0(H)))
\end{align}
for all $H\in\HH_0$. Since $\psi(w)=\Psi^\rat([\cO_w])^*$, the injectivity of the $(\HH_0,\heis)$-homomorphism $\Psi^\rat: \KFlrat\to\bFun^\rat$ gives (1). Then the $\heis$-linear independence of $\{\cO_w\}_{w\in W}$ gives (2).
\end{proof}




%
%

\subsection{Spherical part}\label{SS:sph}

\newcommand{\Ker}{\mathrm{Ker}}

Define
\begin{align}
\Ksph=\bigcap_{i\neq 0}\mathrm{Ker}\,T_i\subset\Kfin.
\end{align} 
Since $T_i=D_i-1$ and $D_i$ is idempotent, we have $\Ker\,T_i=\Im\,D_i$. Using Lemma~\ref{L:Kfin-basis} and \eqref{E:D-act}, one sees that $\Ksph$ has a $\Z[q^{\pm 1}]$-basis given by the classes $\{\cO_{y^\be}(\la)\}_{\be\in Q^\vee,\la\in P}$. Since $\Fl(y^\be) \ (\be\in Q^\vee)$ are exactly the Schubert varieties in $\Flrat$ which are $G(\cR)$-stable, it is reasonable to regard $\Ksph$ as (part of) the $(G(\cR)\rtimes \C^\times)$-equivariant $K$-theory of $\Flrat$. Observe that $\Ksph$ is generated freely as a right $\heis$-module by $[\cO_{\Fl}]$.

It is well known from the theory of affine Hecke algebras that the subalgebra $\Z[X]^W:=Z[X^{\nu+k\delta} : \nu\in P,k\in\Z]^W \subset \HH_0$ commutes with $T_i$ for $i\neq 0$. Thus $\Z[X]^W$ acts on $\Ksph$ and hence for any $H=f(X)\in\Z[X]^W$ we can write
\begin{align}
f(X)\cdot[\cO_{\Fl}]=[\cO_{\Fl}]\cdot \varrho_0^{\sph}(f(X))
\end{align}
for a unique $\varrho_0^{\sph}(f(X))\in\heis$, giving a homomorphism $\varrho_0^\sph : \Z[X]^W\to\heis$.

Let $\pi_{vw}(A)=A_{vw}$ for any $W\times W$-matrix $A$. We deduce that $\varrho_0^{\sph}(f(X))=\pi_{ee}(\varrho_0(f(X)))$ and $\pi_{we}(\varrho_0(f(X)))=0$ for all $w\neq e$. 

Let $\Z[Y]^W=\Z[Y^{\nu+k\delta} : \nu\in P,k\in\Z]^W\subset\HH_0'$. By Theorem~\ref{T:main}, we obtain:

\begin{cor}\label{C:sph}
The following diagram is commutative:
\begin{center}
\begin{tikzcd}
\Z[Y]^W \arrow[r, "\varphi_0^{-1}"] \arrow[d,"\pi_{ee}\circ \varrho_0'"]
& \Z[X]^W \arrow[d, "*\circ\varrho_0^{\sph}"] \\
\heis' \arrow[r, "\tau"]
& \heis
\end{tikzcd}
\end{center}

\end{cor}

\begin{rem}
The map $\pi_{ee}\circ \varrho_0'$ coincides with the $q$-Toda system of difference operators of \cite{C2,CO}.
\end{rem}

\begin{ex}
Recall the matrices $\varrho_0(X^{-\fund})$ and $\varrho_0(X^\fund)$ for $G=SL(2)$ from \eqref{E:X-om} and \eqref{E:Xom}. We have
\begin{align}
\varrho_0(X^{-\fund}+X^\fund) &= 
\begin{bmatrix}
x^{-\fund}+(1-y^{\al^\vee}) x^\fund & x^\fund y^{\al^\vee}-y^{\al^\vee}x^\fund \\
0 & x^{-\fund}+x^\fund (1-y^{\al^\vee})
\end{bmatrix}
\end{align}
and hence
\begin{align}
\varrho_0^\sph(X^{-\fund}+X^\fund) &= x^{-\fund}+(1-y^{\al^\vee}) x^\fund.
\end{align}
\end{ex}

%
%

\end{document}